\newfont{\bcb}{msbm10}
\newfont{\matb}{cmbx10}
\newfont{\got}{eufm10}

\documentclass[12pt]{amsart}
\usepackage{amsmath, amsthm, amscd, amsfonts, amssymb, latexsym, graphicx, color}
\usepackage[bookmarksnumbered, colorlinks, plainpages]{hyperref}

\usepackage[cp1250]{inputenc}

\usepackage{amsmath,amsthm}
\usepackage{amssymb,latexsym}
\usepackage{enumerate}

\newcommand{\matR}{\mathbb{R}}

\newcommand{\matN}{\mathbb{N}}
\newcommand{\matC}{\mathbb{C}}

\newtheorem{theorem}{Theorem}[section]
\newtheorem{lemma}[theorem]{Lemma}

\newtheorem{corollary}[theorem]{Corollary}
\theoremstyle{definition}

\newtheorem{example}[theorem]{Example}

\theoremstyle{remark}
\newtheorem{remark}[theorem]{Remark}
\numberwithin{equation}{section}

\begin{document}

\title[quantifier elimination in quasianalytic structures]
{A counterexample concerning quantifier elimination in quasianalytic structures}

\author[Krzysztof Jan Nowak]{Krzysztof Jan Nowak}



\begin{abstract}
This paper provides an example of a quasianalytic structure which,
unlike the classical analytic structure, does not admit quantifier
elimination in the language of restricted quasianalytic functions
augmented by the reciprocal function $1/x$. It also demonstrates
that \L{}ojasiewicz's theorem that every subanalytic curve is
semianalytic is no longer true in quasianalytic structures. Our
construction applies rectilinearization of terms, established in
our earlier papers, as well as some theorems on power substitution
for Denjoy--Carleman classes and on non-extendability of
quasianalytic function germs. The last result relies on
Grothendieck's factorization and open mapping theorems for
(LF)-spaces.
\end{abstract}

\maketitle

\section{Introduction}

In our earlier papers~\cite{Now2,Now4}, we established, for a
quasianalytic structure,  quantifier elimination as well as
rectilinearization and description of definable functions by
terms, in the language augmented by rational powers. Full
generality was achieved in the latter paper, which constitutes a
basis for a new article being currently prepared. The main purpose
of this paper is to provide a negative answer to the problem,
formulated in~\cite{Now4}, whether a quasianalytic structure
admits quantifier elimination in the language augmented merely by
the reciprocal function $1/x$. In the case of the classical
structure $\mathbb{R}_{an}$, the affirmative answer was given by
J.~Denef and L.~van den Dries~\cite{Denef-Dries}. The construction
of a counterexample, given in Section~4, is based on the
rectilinearization of terms, established in our earlier
paper~\cite{Now3} (see also~\cite{Now4}), as well as on two
function-theoretic results concerning Denjoy--Carleman classes.
The first, presented in Section~2, refers to power substitution
for those classes. The other is a non-extendability theorem, which
is a refinement of a theorem by V.~Thilliez~\cite{Thil-2}.

\vspace{1ex}

Let us recall (cf.~\cite{Now1,Now2,Now5,Now6}) that a
quasianalytic structure $\mathbb{R}_{\mathcal{Q}}$ is the
expansion of the real field by restricted $\mathcal{Q}$-analytic
functions (abbreviated to $\mathcal{Q}$-functions) determined by a
quasianalytic system $\mathcal{Q} =(\mathcal{Q}_{m})_{m\in \matN}$
of sheaves of local $\matR$-algebras of smooth functions on
$\matR^{m}$, subject to the following six conditions:

\begin{enumerate}
  \item each algebra $\mathcal{Q}(U)$ contains the
      restrictions of polynomials;
  \item $\mathcal{Q}$ is closed under composition, i.e.\ the
      composition of $\mathcal{Q}$-maps is a $\mathcal{Q}$-map (whenever it is
      well defined);
  \item $\mathcal{Q}$ is closed under inverse, i.e.\ if
      $\varphi : U \longrightarrow V$ is a $\mathcal{Q}$-map between
      open subsets $U,V \subset \matR^{n}$, $a \in U$, $b \in
      V$ and if $\partial \varphi/\partial x (a) \neq 0$, then
      there are neighbourhoods $U_{a}$ and $V_{b}$ of $a$ and
      $b$, respectively, and a $\mathcal{Q}$-diffeomorphism $\psi: V_{b}
      \longrightarrow U_{a}$ such that $\varphi \circ \psi$ is
      the identity map on $V_{b}$;
  \item $\mathcal{Q}$ is closed under differentiation;
  \item $\mathcal{Q}$ is closed under division by a
      coordinate, i.e.\ if a function $f \in {\mathcal{Q}}(U)$
      vanishes for $x_{i}=a_{i}$, then $f(x)=
      (x_{i}-a_{i})g(x)$ with some $g \in \mathcal{Q}(U)$;
  \item $\mathcal{Q}$ is quasianalytic, i.e.\ if $f \in
      {\mathcal Q}(U)$ and the Taylor series $\widehat{f}_{a}$
      of $f$ at a point $a \in U$ vanishes, then $f$ vanishes
      in the vicinity of $a$.
\end{enumerate}


Note that $\mathcal{Q}$-analytic maps (abbreviated to
$\mathcal{Q}$-maps) give rise, in the ordinary manner, to the
category $\mathcal{Q}$ of $\mathcal{Q}$-manifolds, which is a
subcategory of that of smooth manifolds and smooth maps.
Similarly, $\mathcal{Q}$-analytic, $\mathcal{Q}$-semianalytic and
$\mathcal{Q}$-subanalytic sets can be defined. The above
conditions ensure some (limited) resolution of singularities in
the category $\mathcal{Q}$, including transformation to a normal
crossing by blowing up (cf.~\cite{BM,Rol-Spei-Wil}), upon which
the geometry of quasianalytic structures relies; especially, in
the absence of their good algebraic properties
(cf.~\cite{Rol-Spei-Wil,Now1,Now2,Now4}). Consequently, the
structure $\mathbb{R}_{\mathcal{Q}}$ is model complete and
o-minimal. Its definable subsets coincide with those subsets of
$\matR^{n}$, $n \in \matN$, that are $\mathcal{Q}$-subanalytic in
a semialgebraic compactification of $\matR^{n}$. On the other
hand, every polynomially bounded, o-minimal structure $\mathcal R$
determines a quasianalytic system of sheaves of germs of smooth
functions that are locally definable in $\mathcal R$.

\vspace{1ex}

The examples of such categories are provided by quasianalytic
Denjoy--Carleman classes $\mathcal{Q}_{M}$, where $M=(M_{n})_{n\in
\matN}$ are increasing sequences with $M_{0}=1$. The class
$\mathcal{Q}_{M}$ consists of smooth functions $f(x) =
f(x_{1},\ldots,x_{m})$ in $m$ variables, $m \in \matN$, which are
locally submitted to the following growth condition for their
derivatives:
$$ \left| \partial^{|\alpha|} f
/\partial x^{\alpha} (x) \right| \leq C \cdot
   R^{|\alpha|} \cdot |\alpha|! \cdot M_{|\alpha|} \ \ \
   \mbox{ for all } \ \ \alpha \in \matN^{n}, $$
with some constants $C, R >0 $ depending only on the vicinity of a
given point. This growth condition is often formulated in a
slightly different way:
$$ \left| \partial^{|\alpha|} f
/\partial x^{\alpha} (x) \right| \leq C \cdot
   R^{|\alpha|} \cdot M_{|\alpha|}' \ \ \
   \mbox{ for all } \ \ \alpha \in \matN^{n}, $$
where $M_{n}' = n! M_{n}$. Obviously, the class $\mathcal{Q}_{M}$
contains the real analytic functions.

\vspace{1ex}

In order to ensure some important algebraic and analytic
properties of the class $\mathcal{Q}_{M}$, it suffices to assume
that the sequence $M$ or $M'$ is log-convex. The latter implies
that it is closed under multiplication (by virtue of the Leibniz
formula). The former assumption is stronger and implies that it is
closed under composition (Roumieu~\cite{Rou}) and under inverse
(Komatsu~\cite{Kom}); see also~\cite{BM}. Hence the set
$\mathcal{Q}_{m}(M)$ of germs at $0 \in \mathbb{R}^{m}$ of
$\mathcal{Q}_{M}$-analytic functions is a local ring. Then,
moreover, the class $\mathcal{Q}_{M}$ is quasianalytic iff
$$ \sum_{n=0}^{\infty} \, \frac{M_{n}}{(n+1)M_{n+1}} = \infty $$
(the Denjoy--Carleman theorem; see e.g.~\cite{Ru}), and is closed
under differentiation and under division by a coordinate iff
$$ \sup_{n} \, \sqrt[n]{\left( \frac{M_{n+1}}{M_{n}} \right)} <
   \infty $$
(cf.~\cite{M2,Thil}). It is well-known (cf.~\cite{C,C-M,Thil})
that, given two log-convex sequences $M$ and $N$, the inclusion
$\mathcal{Q}_{m}(M) \subset \mathcal{Q}_{m}(N)$ holds iff there is
a constant $C>0$ such that $M_{n} \leq C^{n}N_{n}$ for all $n \in
\matN$ or, equivalently,
$$ \sup \, \left\{  \sqrt[n]{\frac{M_{n}}{N_{n}}}: \
   n \in \matN \right\} < \infty. $$

\section{Power substitution for Denjoy--Carleman classes}

Consider an increasing sequence $M=(M_{n})$ of real numbers with
$M_{0}=1$. Let $I$ be an interval (open or closed) contained in
$\matR$. By $\mathcal{Q}(I,M)$ we denote the class of functions on
$I$ that satisfy the following growth condition for their
derivatives:
$$ \left| \partial^{|\alpha|} f
/\partial x^{\alpha} (x) \right| \leq C \cdot
   R^{|\alpha|} \cdot |\alpha|! \cdot M_{|\alpha|} \ \ \
   \mbox{ for all } \ \ \alpha \in \matN^{n}, $$
with some constants $C, R >0$.

\vspace{1ex}

The main purpose of this section is to prove the following

\begin{theorem}\label{substitution}
Let $p >1$ be an integer and $I$ the interval $[0,1]$ or $[-1,1]$
according as $p$ is even or odd. Consider power substitution $x =
\xi^{p}$, which is a bijection of $I$ onto itself. Let $f: I
\longrightarrow \matR$ be a smooth function. If
$$ F(\xi) := (f \circ \varphi)(\xi) = f(\xi^{p}) \in
   \mathcal{Q}(I,M), $$
then \ \ $f(x) \in \mathcal{Q}(I,M^{(p)})$, \ where the sequence
$M^{(p)}$ is defined by putting $M^{(p)}_{n}:= M_{pn}$.

Equivalently, in terms of the corresponding sequences $M'$, the
function $f$ belongs to the class determined by the sequence
$M'^{(p)}_{n}:= 1/n^{(p-1)n} \cdot M'_{pn}$, $n \in \mathbb{N}$.
\end{theorem}


\begin{remark}\label{rem1} Using a function constructed by Bang,
we shall show
at the end of this section that, in the case where $p=2$ and the
sequence $M$ is log-convex, $\mathcal{Q}(I,M^{(2)})$ is the
smallest Denjoy--Carleman class containing all those functions
$f(x)$. As communicated to us in written form by E.~Bierstone and
F.V.~Pacheco~\cite{BP}, application of a suitable variation of
Bang's function yields the result for an arbitrary positive
integer $p$.
\end{remark}


\begin{remark}\label{rem2} The case $p=2$ of Theorem \ref{substitution}
may be related to the following problem posed and investigated by
Mandelbrojt~\cite{M1} and \cite[Chap.~VI]{M2}:

\vspace{1ex}

\begin{em}
Consider a smooth function $f(x)$ on the interval $[-1,1]$ and
suppose that $F(\xi):= f(\cos \xi)$ belongs to a class
$\mathcal{Q}(\matR,M)$. To which class on the interval $[-1,1]$
does $f$ belong?
\end{em}
\end{remark}


A complete solution for sequences $M'$ such that $\lim_{n
\rightarrow \infty} \, \sqrt[n]{M'_{n}} = \infty$ was given by
Lalagu\"{e} \cite[Chap.~III]{L}. He proved that those functions
$f$ must belong to the class $\mathcal{Q}(I,N'^{(2)})$,
$N'^{(2)}_{n} = 1/n^{n} \cdot N'_{2n}$, where $N':=M'^{c}$ is the
log-convex regularization of the initial sequence $M'$, and gave a
formula for the smallest Denjoy--Carleman class containing all
those functions $f$. Moreover, the smallest class coincides with
$\mathcal{Q}(I,N'^{(2)})$, if
$$ \liminf_{n \rightarrow \infty} \ \frac{1}{n} \sqrt[n]{M'_{n}} >
   0. $$
The last condition is equivalent to the existence of a constant
$C>0$ such that $M'_{n} \geq n! C^{n}$ for all $n \in \matN$, and
thus implies that the Denjoy-Carleman class corresponding to the
sequence $M$ contains analytic functions.

\vspace{1ex}

Whenever the sequence $M$ is increasing, it is not difficult to
draw Lalagu\"{e}'s result from Theorem \ref{substitution} for the
case $p=2$, because then the functions $f$ from Mandelbrojt's
problem must belong to the class $\mathcal{Q}(I,M^{(2)})$. Indeed,
the proof relies on the observation that the class
$\mathcal{Q}(I,N)$ is closed under analytic substitutions if a
sequence $N$ is increasing. This, in turn, follows immediately
from Cauchy's inequalities and the formula for the derivatives of
a composite function, recalled below:
$$ (f \circ g)^{(n)} (x) = \sum_{k=1}^{n} \, f^{(k)}(g(x))
   \cdot \frac{1}{k!} \cdot \frac{d^{n}}{dX^{n}} \, \left( g(X) - g(x)
   \right)^{k}|_{X=x}. $$


\begin{proof}

Before establishing Theorem \ref{substitution}, we state two
lemmas below.


\begin{lemma}\label{lem1.}
Consider the Taylor expansions
$$ \sum_{i=1}^{\infty} \, \frac{x^{i}}{i} = - \log \, (1-x) \ \
   \mbox{ and } \ \ \left( \sum_{i=1}^{\infty} \, \frac{x^{i}}{i}
   \right) ^{k} \, = \sum_{n=1}^{\infty} \, c_{k,n} x^{n}. $$
Then we have the estimate
$$ c_{k,n} \leq (2e)^{n} \cdot \frac{k!}{n^{k}} \ \ \mbox{ for all } \ \
   k,n \in \matN. $$
\end{lemma}


\begin{proof}
Indeed, it is easy to verify the estimate:
$$ |\log (1-z) | \leq \left| \log 2 + \frac{\pi}{6} \sqrt{-1} \right| \leq 1
   \ \ \mbox{ for all } \ \ z \in \matC, \ |z| \leq 1/2. $$
By Cauchy's inequalities, we thus get $|c_{k,n}| \leq 2^{n}$.
Since $e^{n} > n^{k}/k!$ \ for all $n,k \in \matN$, we have
$$ c_{k,n} \leq 2^{n} < 2^{n} \cdot e^{n} \cdot \frac{k!}{n^{k}} =
    (2e)^{n} \cdot \frac{k!}{n^{k}}, $$
as asserted.
\end{proof}

\vspace{1ex}

As an immediate consequence, we obtain

\vspace{1ex}

\begin{corollary}\label{cor}
$$  \sum_{i_{1}+\ldots+i_{k}=n} \, \frac{1}{i_{1}} \cdot \ldots
   \cdot \frac{1}{i_{k}} = c_{k,n} \leq (2e)^{n} \cdot \frac{k!}{n^{k}}
   \ \ \mbox{ for all } \ \ k,n \in \matN. $$
\end{corollary}

\vspace{1ex}

\begin{lemma}\label{lem2}
Let $p,k \in \matN$ with $p > 1$, $k \geq 1$, and
$$ \alpha_{k}(X,x) := \frac{1}{k!} \left( X^{1/p} - x^{1/p}
   \right)^{k}
   \ \ \mbox{ for } \ \ X,x > 0, $$
where $\alpha_{k}$ is regarded as a function in one variable $X$
and parameter $x$. Then we have the estimate
$$ \left| \alpha_{k}^{(n)}(x,x) \right| \leq (2e)^{n} \cdot n^{n-k}
   \cdot x^{- \frac{pn-k}{p}} \ \ \mbox{ for all } \ \
   n,k \in \matN, \ \ x > 0. $$
\end{lemma}

\begin{proof}
Consider first the case $k=1$, \ $\alpha_{1}(X,x)= X^{1/p} -
x^{1/p}$. Then
$$ \alpha_{1}^{(n)}(x,x) = \pm \frac{(p-1)(2p-1) \cdot \ldots
   \cdot ((n-1)p -1)}{p^{n}} \cdot x^{- \frac{pn-1}{p}}, $$
whence
$$ \left| \alpha_{1}^{(n)}(x,x) \right| \leq (n-1)! \cdot x^{- \frac{pn-1}{p}}
   \leq n^{n-1} \cdot x^{- \frac{pn-1}{p}}, $$
as asserted. The Taylor expansion of $\alpha_{1}(X,x)$ at $X=x$ is
$$ \alpha_{1}(X,x) = \sum_{i=1}^{\infty} \, a_{i} \cdot x^{- \frac{pi-1}{p}}
   (X-x)^{i}, $$
where
$$ a_{i} := \pm \frac{1}{i!} \cdot \frac{(p-1)(2p-1) \cdot \ldots
   \cdot ((i-1)p -1)}{p^{i}}; $$
obviously, $|a_{i}| \leq (i-1)!/i! =1/i$. We thus get
$$ \alpha_{k}(X,x) = \frac{1}{k!} \left( \sum_{i=1}^{\infty} \, a_{i} \cdot x^{- \frac{pi-1}{p}}
   (X-x)^{i} \right)^{k} = \sum_{j=1}^{\infty} \, b_{j} \cdot x^{- \frac{pj-k}{p}}
   (X-x)^{j},  $$
where
$$ b_{j} := \frac{1}{k!} \sum_{i_{1}+\ldots+i_{k}=j} \, a_{i_{1}}
   \cdot \ldots \cdot a_{i_{k}}. $$
Then
$$ |b_{n} | \leq \frac{1}{k!} \sum_{i_{1}+\ldots+i_{k}=n} \,
   |a_{i_{1}}| \cdot \ldots \cdot |a_{i_{k}}| \leq \frac{1}{k!}
   \sum_{i_{1}+\ldots+i_{k}=n} \, \frac{1}{i_{1}} \cdot \ldots
   \cdot \frac{1}{i_{k}} = $$
$$  = \frac{c_{k,n}}{k!} \leq \frac{(2e)^{n}}{n^{k}}
   \ \ \mbox{ for all } \ \ n,k \in \matN; $$
the last inequality follows from Corollary \ref{cor}. Hence
$$ \left| \alpha_{k}^{(n)}(x,x) \right| \leq n! \cdot | b_{n}|
   \cdot x^{- \frac{pn-k}{p}} \leq (2e)^{n} \cdot \frac{n!}{n^{k}}
   \cdot x^{- \frac{pn-k}{p}} \leq (2e)^{n} \cdot n^{n-k} \cdot x^{- \frac{pn-k}{p}} $$
for all $n,k \in \matN$, $x > 0$, as asserted.
\end{proof}

\vspace{2ex}

Now we can readily pass to the proof of Theorem
\ref{substitution}. We shall work with estimates corresponding to
the sequence $M'$. So suppose that
$$ \left| F^{(n)}(\xi) \right| \leq A^{n} M'_{n} $$
for all $n \in \matN$, $\xi \in I$ and some constant $A>0$. We are
going to estimate the growth of the $n$-th derivative $f^{(n)}$.
Fix $n \in \matN$ and put:
$$ p_{n}(x) := T_{0}^{n}f(x) = \sum_{k=0}^{n-1} \, f^{k}(0)
   \frac{x^{k}}{k!}, \ \ \ r_{n}(x) := f(x) - p_{n}(x), $$
$$ P_{n}(\xi) := p_{n}(\xi^{p}) \ \ \mbox{ and } \ \ R_{n}(\xi) :=
   r_{n}(\xi^{p}). $$
Obviously,
$$ P_{n}^{(pn)} \equiv 0 \ \ \mbox{ and } \ \ R_{n}^{(pn)} \equiv
   F^{(pn)}. $$
From the Taylor formula, we therefore obtain the estimate
$$ \left| R_{n}^{(pn)}(\xi) \right| \leq A^{pn} M'_{pn} \cdot
   \frac{\xi^{pn-k}}{(pn-k)!} $$
for all $k < pn$, $\xi \in I$. We still need an elementary
inequality
\begin{equation}\label{element}
   \frac{1}{(pn-k)!} \leq \frac{e^{pn}}{n^{pn-k}} \ \ \mbox{ for
   all } \ \ k<pn,
\end{equation}
which can be proven via the following well-known Stirling formula
$$ \sqrt{2\pi n} \cdot \frac{n^{n}}{e^{n}} < n! < e \sqrt{n}
   \cdot \frac{n^{n}}{e^{n}}. $$
Indeed, it suffices to show that
$$ \frac{e^{pn-k}}{(pn-k)^{pn-k}} \leq \frac{e^{pn}}{n^{pn-k}}. $$
When $pn-k \geq n$ or, equivalently, $k \leq (p-1)n$, the last
inequality is evident. Suppose thus that $pn-k<n$ or,
equivalently, $k>(p-1)n$. This inequality is, of course,
equivalent to
$$ \left( \frac{n}{pn-k} \right) ^{pn-k} \leq e^{k}, $$
which holds as shown below:
$$ \left( \frac{n}{pn-k} \right) ^{pn-k} =
   \left( 1 + \frac{k-(p-1)n}{pn-k} \right)^{pn-k} = $$
$$  = \left[ \left( 1 + \frac{k-(p-1)n}{pn-k}
   \right)^{\frac{pn-k}{k-(p-1)n}} \right] ^{k-(p-1)n}< e^{k}. $$

\vspace{1ex}

Now, the foregoing estimate along with inequality \eqref{element}
yield
$$ \left| R_{n}^{(pn)}(\xi) \right| \leq A^{pn} M'_{pn} \cdot \frac{e^{pn}}{n^{pn-k}}
    \cdot \xi^{pn-k}. $$
Applying the formula for the derivatives of a composite function,
we obtain
$$ r_{n}^{(n)}(x) = \sum_{k=1}^{n} \, R_{n}^{(k)}(\xi) \cdot
   \alpha_{k}^{(n)}(x,x). $$
Hence and by Lemma \ref{lem2}, we get
$$ \left| f^{(n)}(x) \right| = \left| r_{n}^{(n)}(x) \right| \leq
   \sum_{k=1}^{n} \, A^{pn} M'_{pn} \cdot \frac{e^{pn}}{n^{pn-k}}
   \cdot |\xi|^{pn-k} \cdot (2e)^{n} n^{n-k} \cdot |\xi|^{-(pn-k)} = $$
$$ = n \cdot (2e)^{n} \cdot (eA)^{pn} \cdot \frac{M'_{pn}}{n^{(p-1)n}}, $$
which completes the proof of Theorem \ref{substitution}.
\end{proof}

\vspace{2ex}

Finally, we show that, when the sequence $M'$ is log-convex,
$\mathcal{Q}(I,M^{(2)})$ is the smallest Denjoy--Carleman class
containing all smooth functions $f(x)$ on the interval $I =[0,1]$
such that $F(\xi) = f(\xi^{2}) \in \mathcal{Q}(I,M)$. We make use
of a classical function constructed by Bang~\cite{Ba}, applied in
his proof that the classes determined by log-convex sequences
contain functions with sufficiently large derivatives (the result
due to Cartan~\cite{C} and Mandelbrojt~\cite{C-M}; see also
\cite{Thil}, Section~1, Theorem~1).

\vspace{1ex}

The logarithmic convexity of the sequence $M'$ yields for every
$j,k \in \matN$ the inequality
$$ \left( \frac{1}{m_{k}} \right)^{k-j} \leq \frac{M'_{j}}{M'_{k}} \ \
   \mbox{ where } \ \ m_{k} := \frac{M'_{k+1}}{M'_{k}}. $$
Consequently,
$$ F(\xi) := \sum_{k=0}^{\infty} \, \frac{M'_{k}}{(2m_{k})^{k}} \cos
   (2m_{k}\xi) $$
is an even smooth function on $\matR$ such that
$$ F(\xi) \in \mathcal{Q}(\matR,M) \ \ \mbox{ and } \ \
   \left| F^{(2n)}(0) \right| \geq M'_{2n} $$
for all $n \in \matN$. Therefore $F(\xi) = f(\xi^{2})$ for some
smooth function $f$ on $\mathbb{R}$ (cf.~\cite{Whi}), and we get
$$ f^{(n)}(0) = \frac{n!}{(2n)!} F^{(2n)}(0) \ \ \mbox{ and } \ \
   \left| f^{(n)}(0) \right| \geq \frac{n! M'_{2n}}{(2n)!}, $$
which is the desired result.

\vspace{1ex}

It was communicated to us by E.~Bierstone and
F.V.~Pacheco~\cite{BP} that, in order to obtain this result for an
arbitrary positive integer $p$, one can replace $\cos x$ by the
function
$$ C_{p}(x) := \sum_{j=0}^{\infty} \frac{x^{jp}}{(jp)!} $$
with the properties $C_{p}^{(pn)}(x)=C_{p}(x)$,
$C_{p}^{(pn)}(0)=1$ and $\left| C_{p}^{(n)}(x) \right| \leq e$ for
all $n \in \matN$, $x \in [-1,1]$. Then
$$ F(\xi) := \sum_{k=0}^{\infty} \, \frac{M'_{k}}{(2m_{k})^{k}}
   C_{p}(2m_{k}\xi) $$
is a smooth function on $\matR$ such that
$$ F(\xi) \in \mathcal{Q}([-1,1],M) \ \ \mbox{ and } \ \
   \left| F^{(pn)}(0) \right| \geq M'_{pn} $$
for all $n \in \matN$. As before, there exists a smooth function
$f$ on $\mathbb{R}$ such that
$$ F(\xi) = f(\xi^{p}), \ \ f^{(n)}(0) = \frac{n!}{(pn)!} F^{(pn)}(0) \ \ \mbox{ and } \ \
   \left| f^{(n)}(0) \right| \geq \frac{n! M'_{pn}}{(pn)!}. $$

\vspace{2ex}

We conclude this section with some examples, one of which (namely,
for $k=2$) will be applied to the construction of our
counterexample in the last section.

\vspace{2ex}

\begin{example}\label{example} Fix an integer $k \in \matN$, $k \geq 1$, and put
$$ \log^{(k)} := \underbrace{\log \circ \ldots \circ \log}_{k
   \mbox{ \scriptsize times}}, \ \ \mbox{ and } \ \
   e\uparrow \uparrow k := (\underbrace{\exp \circ \ldots \circ \exp}_{k
   \mbox{ \scriptsize times}}) (1); $$
let $n_{k}$ be the smallest integer greater than $e \uparrow
\uparrow k$. Then the sequence
$$ \left( \log^{(k)} n \right)^{n} \ \ \mbox{ for } \ \ n \geq n_{k} $$
is log-convex. Further, the shifted sequence:
$$ M=(M_{n}), \ \ M_{n} := \frac{1}{\left( \log^{(k)} n_{k} \right)^{n_{k}} } \cdot
   \left( \log^{(k)} (n_{k} +n) \right)^{(n_{k}+n)}, $$
determines a quasianalytic class closed under derivatives; the
former follows from Cauchy's condensation criterion. It is easy to
check that the sequences $M^{(p)}$, $p>1$, are quasianalytic when
$k>1$, but are not quasianalytic when $k=1$.
\end{example}

\section{Non-extendability of quasianalytic germs}

In this section we are concerned with a result by
V.~Thilliez~\cite{Thil-2} on the extension of quasianalytic
function germs in one variable, recalled below. As before,
consider two log-convex sequences $M$ and $N$ with $M_{0}=N_{0}=1$
such that $\mathcal{Q}_{1}(M) \subset \mathcal{Q}_{1}(N)$. Denote
by $\mathcal{Q}_{1}(M)^{+}$ the local ring of right-hand side
germs at zero (i.e.\ germs of functions from
$\mathcal{Q}([0,\varepsilon],M)$ for some $\varepsilon >0$).

\vspace{2ex}

\begin{theorem}\label{non1}
If \ $\mathcal{Q}_{1}(N)$ is a quasianalytic local ring, then
$$ \mathcal{O}_{1} \varsubsetneq \mathcal{Q}_{1}(M) \subset
   \mathcal{Q}_{1}(N) \ \ \Longrightarrow \ \
   \mathcal{Q}_{1}(M)^{+} \setminus \mathcal{Q}_{1}(N) \neq \emptyset, $$
i.e.\ there exist right-hand side germs from
$\mathcal{Q}_{1}(M)^{+}$ which do not extend to germs from
$\mathcal{Q}_{1}(N)$. Here $\mathcal{O}_{1}$ stands for the local
ring of analytic function germs in one variable at $0 \in
\mathbb{R}$.
\end{theorem}


\begin{remark}\label{rem3} Theorem \ref{non1} may be related to
the research by M.~Langenbruch~\cite{Lb} on the extension of
ultradifferentiable functions in several variables, principally
focused on the non-quasianalytic case, which seems to be more
difficult in this context. His extension problem is, roughly
speaking, as follows:

\vspace{1ex}

\begin{em}
Given two compact convex subsets $K,K_{1}$ of $\matR^{m}$ such
that $\mbox{\rm int}\, (K) \neq \emptyset$ or $K = \{ 0 \}$ and $K
\subset \mbox{\rm int}\, (K_{1})$, characterize the sequences of
positive numbers $M$ and $N$ such that every function from the
class $\mathcal{Q}(K,M)$ extends to a function from
$\mathcal{Q}(K_{1},N)$.
\end{em}
\end{remark}


M.~Langenbruch applies, however, different methods and techniques
in comparison with V.~Thilliez. In particular, his approach is
based on the theory of Fourier transform and plurisubharmonic
functions.

\vspace{1ex}

On the other hand, Thilliez's approach relies on Grothendieck's
version of the open mapping theorem (cf.~\cite{Gro}, Chap.~4,
Part~1, Theorem~2 or \cite{M-V}, Part~IV, Chap.~24) and Runge
approximation. It also enables a formulation of the
non-extendability theorem for quasianalytic function germs on a
compact convex subset $K \subset \matR^{m}$ with $0 \in K$.

\vspace{1ex}

Nevertheless, in order to construct our counterexample in the next
section, we need a refinement of Theorem \ref{non1}, stated below.
Thilliez's proof can be adapted mutatis mutandis. We shall outline
it for the reader's convenience. Consider an increasing countable
family $M^{[p]}$, $p\in \matN$, of log-convex sequences, i.e.\
$$ 1= M^{[p]}_{0} \leq M^{[p]}_{1} \leq M^{[p]}_{2} \leq M^{[p]}_{3} \leq
   \ldots\ \ \ \mbox{ and } \ \ M^{[p]}_{j} \leq M^{[q]}_{j} $$
for all $j,p,q \in \matN$, $p \leq q$. Then we receive an
ascending sequence of local rings
$$ \mathcal{Q}_{1}(M^{[1]}) \subset \mathcal{Q}_{1}(M^{[2]}) \subset
   \mathcal{Q}_{1}(M^{[3]}) \subset \ldots $$
such that $\mathcal{Q}_{1}(M^{[p]})$ is dominated by
$\mathcal{Q}_{1}(M^{(q)})$ for all $p,q \in \matN$ with $p \leq
q$.

\vspace{2ex}

\begin{theorem}\label{non2}
If every local ring $\mathcal{Q}_{1}(M^{[p]})$ is quasianalytic,
then
$$ \mathcal{O}_{1} \varsubsetneq \mathcal{Q}_{1}(M) \subset \bigcup_{p \in \matN} \,
   \mathcal{Q}_{1}(M^{[p]}) \ \ \Longrightarrow \ \
   \mathcal{Q}_{1}(M)^{+} \setminus \bigcup_{p \in \matN} \,
   \mathcal{Q}_{1}(M^{[p]}) \neq \emptyset. $$
\end{theorem}


\begin{proof}
We adopt the following notation. For a smooth function $f$ on an
interval $I \subset \matR$ and $r>0$, put
$$ \| f \|_{M,I,r} := \sup \, \left\{ \frac{\left| f^{(n)}(x) \right|}{r^{n} n!
   M_{n}}: \ n \in \matN, \ x \in I \right\}. $$
For $k \in \matN$, $k>0$, let \ $ B_{k}(M)$ \ or \ $B_{k}(M)^{+}$,
respectively, denote the Banach space with norm
$$ \| \cdot \|_{M,[-1/k,1/k],k}  \ \ \mbox{ or } \ \ \| \cdot \|_{M,[0,1/k],k}, $$
of those smooth functions on the interval $[-1/k,1/k]$ or
$[0,1/k]$) such that
$$ \| f \|_{M,[-1/k,1/k],k} < \infty \ \ \mbox{ or } \ \
  \| f \|_{M,[0,1/k],k} < \infty, \ \ \mbox{respectively.} $$

As the canonical embeddings
$$ B_{k}(M) \hookrightarrow B_{l}(M) \ \ \mbox{ and } \ \
   B_{k}(M)^{+} \hookrightarrow B_{l}(M)^{+}, \ \ k,l \in \matN, \ k \leq l, $$
are compact linear operators (a consequence of Ascoli's theorem;
cf.~\cite{Kom0}), one can endow the local rings
$\mathcal{Q}_{1}(M)$ and $\mathcal{Q}_{1}(M)^{+}$ with the
inductive topologies. Similarly, the countable union of local
rings
$$ \bigcup_{p \in \matN} \, \mathcal{Q}_{1}(M^{[p]}) $$
is the inductive limit of the sequence $B_{k}(M^{(k)})$, $k \in
\matN$, $k>0$, of Banach algebras. Further, the local ring
$$ \bigcup_{p \in \matN} \, \mathcal{Q}_{1}(M^{[p]}) \cap
   \mathcal{Q}_{1}(M)^{+} $$
is the inductive limit of the sequence
$$ B_{k}(M^{(k)}) \cap B_{k}(M)^{+}, \ \ \ k \in \matN, \ k>0, $$
of Banach algebras with norms
$$ \| f \|_{k} := \max \, \{ \| f \|_{M^{[k]},[-1/k,1/k],k}\, , \
   \| f \|_{M,[0,1/k],k} \}. $$
Clearly, the restriction operator
$$ R: \bigcup_{p \in \matN} \, \mathcal{Q}_{1}(M^{[p]}) \cap
   \mathcal{Q}_{1}(M)^{+} \longrightarrow \mathcal{Q}_{1}(M)^{+} $$
is continuous and injective by quasianalyticity. We must show that
$R$ is not surjective.

\vspace{1ex}

To proceed with {\em reduction ad absurdum}, suppose that $R$ is
surjective. By Grothendieck's version of the open mapping theorem
(cf.~\cite{Gro}, Chap.~4, Part~1, Theorem~2 or \cite{M-V},
Part~IV, Chap.~24), the operator $R$ is a homeomorphism onto the
image. Further, by Grothendieck's factorization theorem ({\em
loc.~cit.\/}), for each $k \in \matN$ there is an $l \in \matN$
and a constant $C>0$ such that
$$ R \left( B_{l}(M^{(l)}) \right) \supset R \left( B_{l}(M^{(l)}) \cap
   B_{l}(M)^{+} \right) \supset B_{k}(M)^{+}, $$
and
$$ \| R^{-1} f \|_{M^{(l)},[-1/l,1/l],l} \: \leq \:
   \| R^{-1} f \|_{l} \: \leq \: C \| f \|_{M,[0,1/k],k} $$
for all $f \in B_{k}(M)^{+}$. In particular, there is an $l \in
\matN$ and a constant $A>0$ such that
$$ R \left( B_{l}(M^{(l)}) \right) \supset B_{1}(M)^{+}, $$
and
$$ \| R^{-1} f \|_{M^{(l)},[-1/l,1/l],l} \: \leq \:
   A \| f \|_{M,[0,1],1} $$
for all $f \in B_{1}(M)^{+}$. In particular,
$$ \left| P\left(-\frac{1}{l}\right) \right| \leq A \| P
   \|_{M,[0,1],1} $$
for every polynomial $P \in \matC[x]$. Put
$$ W := \left\{ z \in \matC: \: \mbox{dist}\, (z,[0,1]) \leq \frac{1}{2l} \right\}
$$
and
$$ B := \sup \, \left\{  \frac{(2l)^{n}}{M_{n}}: \: n \in
   \matN \right\} < \infty; $$
the last inequality holds because \ $\mathcal{O}_{1} \varsubsetneq
\mathcal{Q}_{1}(M)$ \ whence
$$ \sup \, \left\{  \sqrt[n]{M_{n}}: \  n \in \matN \right\} = \infty. $$
It follows from Cauchy's inequalities that
$$ \sup \, \left\{ \left| P^{(n)}(x) \right| : x \in [0,1] \right\} \leq n!
   (2l)^{n} \sup \, \{ |P(x)|: x \in W \}, $$
and hence
$$ \| P \|_{M,[0,1],1} \leq B \sup \, \{ |P(x)|: x \in W \}.  $$
Consequently,
$$ \left| P\left(-\frac{1}{l}\right) \right| \leq AB \sup \, \{ |P(x)|: x \in W \} $$
for every polynomial $P \in \matC [x]$. But, by virtue of Runge
approximation, there exists a sequence of polynomials $P_{\nu} \in
\matC[x]$ which converges uniformly to $0$ on $W$, and to $1$ for
$x=-1/l$. This contradicts the above estimate, and thus the
theorem follows.
\end{proof}

\section{Construction of a counterexample}

We first recall some results on rectilinearization of terms from
our paper~\cite{Now3} (see also~[Section~4]~\cite{Now4}). We begin
with some suitable terminology. By a quadrant in $\matR^{m}$ we
mean a subset of $\matR^{m}$ of the form:
$$ \{ x = (x_{1}, \ldots, x_{m}) \in \matR^{m}: \; $$
$$ x_{i} =0, \, x_{j}>0, \, x_{k}<0 \ \ \mbox{ for } \ i \in
   I_{0}, \, j \in I_{+}, \, k \in I_{-} \}, $$
where $\{ I_{0}, I_{+}, I_{-} \}$ is a disjoint partition of $\{
1,\ldots,m \}$; its trace $Q$ on the cube $[-1,1]^{m}$ shall be
called a bounded quadrant. The interior $\mbox{Int}\,(Q)$ of the
quadrant $Q$ is its trace on the open cube $(-1,1)^{m}$. A bounded
closed quadrant is the closure $\overline{Q}$ of a bounded
quadrant $Q$, i.e.\ a subset of $\matR^{m}$ of the form:
$$ \overline{Q} := \{ x = (x_{1},\ldots,x_{m}) \in [-1,1]^{m}: \; $$
$$ x_{i} =0, \, x_{j} \geq 0, \, x_{k} \leq 0 \ \ \mbox{ for }
   \ i \in I_{0}, \, j \in I_{+}, \, k \in I_{-} \} . $$

By a normal crossing on a bounded quadrant $Q$ in $\matR^{m}$ we
mean here a function $g$ of the form
$$ g(x) = x^{\alpha} \cdot u(x), $$
where $\alpha \in \matN^{m}$ and $u$ is a function
$\mathcal{Q}$-analytic near $\overline{Q}$ which vanishes nowhere
on $\overline{Q}$. The announced results concerning
rectilinearization are the quasianalytic versions of Theorems~2
and ~3 on rectilinearization of terms from our
paper~\cite[Section~2]{Now3}, stated below (see
also~\cite[Section~4]{Now4}). Let $\mathcal{L}$ be the language of
restricted $\mathcal{Q}$-analytic functions augmented by the names
of the reciprocal function $1/x$ and roots.


\begin{theorem}\label{recti1}
If $f: \matR^{m} \longrightarrow \matR$ is a function given
piecewise by a finite number of $\mathcal{L}$-terms, and $K$ is a
compact subset of $\matR^{m}$, then there exists a finite
collection of modifications
$$ \varphi_{i}: [-1,1]^{m} \longrightarrow \matR^{m}, \ \ \ \
i=1,\ldots,p, $$ such that

1) each $\varphi_{i}$ extends to a $\mathcal{Q}$-map in a
neighbourhood of the cube $[-1,1]^{m}$, which is a composite of
finitely many local blow-ups with smooth $\mathcal{Q}$-analytic
centers and power substitutions;

2) the union of the images $\varphi_{i}((-1,1)^{m})$,
$i=1,\ldots,p$, is a neighbourhood of $K$.

3) for every bounded quadrant $Q_{j}$, $j=1,\ldots,3^{m}$, the
restriction to $Q_{j}$ of each function $f \circ \varphi_{i}$,
$i=1,\ldots,p$, either vanishes or is a normal crossing or a
reciprocal normal crossing on $Q_{j}$.
\end{theorem}

\begin{theorem}\label{recti2}
Let $U \subset \matR^{m}$ be a bounded open subset given piecewise
by $\mathcal{L}$-terms, and $f: U \longrightarrow \matR$ be a
function given piecewise by a finite number of
$\mathcal{L}$-terms. Then there exists a finite collection of
modifications
$$ \varphi_{i}: [-1,1]^{m} \longrightarrow \matR^{m}, \ \ \ \
i=1,\ldots,p, $$ such that

1) each $\varphi_{i}$ extends to a $\mathcal{Q}$-map in a
neighbourhood of the cube $[-1,1]^{m}$, which is a composite of
finitely many local blow-ups with smooth $\mathcal{Q}$-analytic
centers and power substitutions;

2) each set $\varphi_{i}^{-1}(U)$ is a finite union of bounded
quadrants in $\matR^{m}$;

3) each set $\varphi_{i}^{-1}(\partial U)$ is a finite union of
 bounded closed quadrants in $\matR^{m}$ of dimension $m-1$;

4) $U$ is the union of the images $\varphi_{i}(\mbox{\rm Int}\,
(Q))$ with $Q$ ranging over the bounded quadrants contained in
$\varphi_{i}^{-1}(U)$, $i=1,\ldots,p$;

5) for every bounded quadrant $Q$,  the restriction to $Q$ of each
function $f \circ \varphi_{i}$ either vanishes or is a normal
crossing or a reciprocal normal crossing on $Q$, unless
$\varphi_{i}^{-1}(U) \cap Q = \emptyset$. \hspace{\fill} $\Box$
\end{theorem}


\begin{remark}\label{rem4} The above theorems can be formulated for
simultaneous rectilinearization of several functions
$f_{1},\ldots,f_{s}$ given by terms. Also, a stronger version
({\em op.~cit.\/})\ holds if the function $f$ is continuous.
\end{remark}

\begin{remark}\label{rem5} Actually, if the function
$f$ is given piecewise by terms in the language of restricted
$\mathcal{Q}$-analytic functions augmented only by the reciprocal
function $1/x$, then one can require ({\em op.~cit.\/}) that the
modifications $\varphi_{i}$ in the above two theorems,
$i=1,\ldots,p$, be a composite of finitely many local blow-ups
with smooth $\mathcal{Q}$-analytic centers.
\end{remark}

\vspace{1ex}

{\bf Counterexample.} Now, we can readily give a counterexample
indicating that quasianalytic structures, unlike the classical
structure $\mathbb{R}_{an}$, may not admit quantifier elimination
in the language augmented merely by the reciprocal function $1/x$.
The example we construct is a plane curve through $0 \in
\matR^{2}$ which is definable in the quasianalytic structure
corresponding to the log-convex sequence
$$ M=(M_{n}), \ \ M_{n} := \frac{1}{\left( \log \log \, 3 \right)^{3} } \cdot
   \left( \log \log \, (n +3) \right)^{(n+3)}; $$
this sequence determines a quasianalytic class closed under
derivatives (cf.~example \ref{example}). By Theorem \ref{non2}, we
can take a function germ
$$ f \in \mathcal{Q}_{1}(M)^{+} \, \setminus \, \bigcup_{p \mbox{ \scriptsize odd}} \,
   \mathcal{Q}_{1}(M^{(p)}). $$
Let $V \subset \matR^{2}$ be the graph of a representative of this
germ in a right-hand side neighbourhood $[0,\varepsilon]$.

\vspace{1ex}

To proceed with {\em reductio ad absurdum\/}, suppose $V$ is given
by a term in the language of restricted $\mathcal{Q}_{M}$-analytic
functions augmented merely by the reciprocal function $1/x$.
Taking into account Remark~\ref{rem5}, we can thus deduce from
Theorem~\ref{recti2} that there would exist a rectilinearization
of this term by a finite sequence of blow-ups of the real plane at
points. Consequently, the germ of $V$ at zero would be contained
in the image $\varphi([-1,1])$, where
$$ \varphi =(\varphi_{1},\varphi_{2}): [-1,1]
   \longrightarrow \matR^{2}, \ \ \varphi(0)=0, $$
is a $\mathcal{Q}_{M}$-analytic homeomorphism. But then the order
of $\varphi_{1}$ at zero must be odd, and thus the set $V$ would
have a parametrization near zero of the form $(\xi^{p}, g(\xi))$,
where $p$ is an odd positive integer and $g$ is a
$\mathcal{Q}_{M}$-function in the vicinity of zero. Hence and by
Theorem \ref{substitution}, we would get
$$ f(x) = g(x^{1/p}) \in \mathcal{Q}_{1}(M^{(p)}), $$
which is a contradiction.


\begin{remark} By virtue of Puiseux's theorem for definable functions
(cf.~\cite{Now5}, Section~2), the germ of every smooth function in
one variable that is definable in the structure
$\mathbb{R}_{\mathcal{Q}_{M}}$ belongs to
$\mathcal{Q}_{1}(M^{(p)})$ for some positive integer $p$.
Therefore the structure under study will not admit quantifier
elimination, even considered with the richer language of
restricted definable quasianalytic functions augmented by the
reciprocal function $1/x$.
\end{remark}

\begin{remark}
Also, our counterexample demonstrates that the classical theorem
of \L{}ojasiewicz~\cite{Loj} that every subanalytic set of
dimension $\leq 1$ is semianalytic is no longer true in
quasianalytic structures.
\end{remark}

\vspace{2ex}

{\bf Acknowledgements.} This research was partially supported by
Research Project No.\ N N201 372336 from the Polish Ministry of
Science and Higher Education.


\vspace{4ex}

\begin{small}
Institute of Mathematics, \ Jagiellonian University

ul.~Profesora \L{}ojasiewicza 6, \ 30-348 Krak\'{o}w, Poland

e-mail:
\email{\textcolor[rgb]{0.00,0.00,0.84}{nowak@im.uj.edu.pl}}

\end{small}

\end{document}